\documentclass{amsart}
\usepackage{amsmath,amssymb,amscd}
\newtheorem{theorem}{Theorem}
\newcommand{\bt}{\begin{theorem}}
\newcommand{\et}{\end{theorem}}

\newtheorem{lemma}{Lemma}
\newcommand{\bl}{\begin{lemma}}
\newcommand{\el}{\end{lemma}}

\newtheorem{corollary}{Corollary}
\newcommand{\bc}{\begin{corollary}}
\newcommand{\ec}{\end{corollary}}

\newcommand{\bpf}{\begin{proof}}
\newcommand{\epf}{\end{proof}}
\newcommand{\beq}{\begin{equation}}
\newcommand{\eeq}{\end{equation}}
\newcommand{\benum}{\begin{enumerate}}
\newcommand{\eenum}{\end{enumerate}}

\newcommand{\N}{\ensuremath{ \mathbf N }}

\title[Supersequences and rearrangements]
{Supersequences, rearrangements of sequences, and the spectrum of bases in additive number theory}
\author{Melvyn B. Nathanson}
\address{Department of Mathematics\\ 
Lehman College (CUNY)\\ 
Bronx, New York 10468}
\email{melvyn.nathanson@lehman.cuny.edu}
\thanks{This work was supported in part by the PSC-CUNY Research Award Program.}



\date{\today}

\begin{document}

\begin{abstract}
The set $ \mathcal{A} = \{a_n \}_{n=1}^{\infty}$  of nonnegative integers is an asymptotic basis 
of order $h$ if every sufficiently large integer can be represented as the sum of $h$ elements of $\mathcal{A}.$
If $a_n \sim \alpha n^h$ for some real number $\alpha > 0,$ 
then $\alpha$ is called an \emph{additive eigenvalue of order $h$}.
The \emph{additive spectrum of order $h$} is the set $\mathcal{N}(h)$ 
consisting of all additive eigenvalues of order $h$.  It is proved that 
there is a positive number $\eta_h \leq 1/h!$ such that $\mathcal{N}(h) = (0, \eta_h)$ 
or $\mathcal{N}(h) = (0, \eta_h].$  The proof uses results about the construction 
of supersequences of sequences with prescribed asymptotic growth, and also 
about the asymptotics of rearrangements of infinite sequences.  
For example, it is proved that there does not exist a strictly increasing sequence 
of integers $B = \{b_n\}_{n=1}^{\infty}$ such that $b_n \sim 2^n$ and $B$ contains 
a subsequence $\{b_{n_k}\}_{k=1}^{\infty}$ such that $b_{n_k} \sim 3^k.$

\end{abstract}

\maketitle

\section{The additive spectrum}
This paper is motivated by the following problem in additive number theory.  
Let  $\mathcal{A}$ be a set of nonnegative integers.
The \emph{counting function} of $\mathcal{A}$ is the function
\[
\mathcal{A}(y,x) = \sum_{\substack{a\in \mathcal{A} \\ y \leq a \leq x}} 1.
\]
In particular, the function 
\[
\mathcal{A}(0,x) = \sum_{\substack{a\in \mathcal{A} \\ 0 \leq a \leq x}} 1 
\]
counts the number of nonnegative elements of the set $\mathcal{A}$ that do not exceed $x$.

Let $h\mathcal{A}$ denote the set of all sums of $h$ not necessarily distinct elements of $\mathcal{A}$.  
The set $\mathcal{A}$ of nonnegative integers is called 
an \emph{asymptotic basis of order $h$} if the sumset $h{\mathcal{A}}$ contains all sufficiently large integers.

Let $x$ be a  real number, and let $[x]$ denote the integer part of $x.$  
The following counting argument shows that 
if $\mathcal{A}$ is an asymptotic basis of order $h$, 
then $\mathcal{A}(x) \gg x^{1/h}$ for all $x \geq 1.$  
If an integer $n \leq x$ is the sum of $h$ nonnegative integers, 
then each summand is at most $x$.  The number of combinations 
with repetitions allowed of $h$ elements of ${ \mathcal{A}} \cap [0,x]$ is 
${\mathcal{A}(0,x)+h-1 \choose h}.$
If $h{\mathcal{A}}$ contains all integers $n \geq n_0,$ then 
\beq   \label{SeqR:GrowthIneq}
x-n_0 < [x]-n_0+1 \leq {\mathcal{A}(0,x)+h-1 \choose h} \leq \frac{ (\mathcal{A} (0,x)+h-1)^h}{h!}
\eeq
and so 
\[
\liminf_{x\rightarrow\infty} \frac{{\mathcal{A}}(0,x)}{x^{1/h}} \geq (h!)^{1/h}.
\]
Let ${\mathcal{A}} = \{ a_n : n=1,2,\ldots\},$ where $a_n < a_{n+1}$ for all $n \geq 1.$  
Then $\mathcal{A}(0,a_n) = n$.  
Replacing $x$ by $a_n$ in inequality~\eqref{SeqR:GrowthIneq}, we obtain
\[
a_n - n_0 < \frac{ (n+h-1)^h}{h!}
\]
and so
\beq  \label{SeqR:MaxEigenvalue}
\limsup_{n\rightarrow\infty} \frac{a_n}{n^h} \leq \frac{1}{h!}.
\eeq

The asymptotic basis 
$\mathcal{A}$ of order $h$ is called \emph{thin} if $\mathcal{A}(0,x) \ll x^{1/h}.$
Equivalently, if  ${\mathcal{A}} = \{ a_n : n=1,2,\ldots\},$ where $a_n < a_{n+1}$ 
for all $n \geq 1,$
then the asymptotic basis $\mathcal{A}$ is thin if and only if there exist 
positive numbers $c_1$ and $c_2$ such that
\[
c_1 n^h \leq a_n \leq c_2 n^h
\]
for all $n$.  
The first examples of thin bases were discovered by Raikov~\cite{raik37} and St\" ohr~\cite{stoh37}, 
and recent constructions are due to Blomer~\cite{blom03}, Hofmeister~\cite{hofm01}, 
and Jia and Nathanson~\cite{jia-nath89}.

Cassels~\cite{cass57} constructed a beautiful 
family of asymptotic bases of order $h$ such that 
\[
a_n \sim \alpha n^h 
\]
that is, $\lim_{n\rightarrow\infty} a_n/ n^h = \alpha.$  
Grekos, Haddad, Helou, and Pikho~\cite{grek-hadd-helo-pihk06} have produced some variations on Cassels' work.
We call the positive real number $\alpha$ 
an \emph{additive eigenvalue of order $h$}, 
and we denote by $\mathcal{N}(h)$ the set of all additive eigenvalues of order $h.$ 
The set $\mathcal{N}(h)$ is called the \emph{additive spectrum of order $h$.}  
We shall prove that if $\alpha$ is an additive eigenvalue of order $h$ 
and if $0 < \beta < \alpha,$ then $\beta$ is also an additive eigenvalue 
of order $h$.  Equivalently, the additive spectrum  $\mathcal{N}(h)$ is an interval 
of the form $(0, \eta_h)$ or $(0, \eta_h],$ where $\eta_h \leq 1/h!$ 
by inequality~\eqref{SeqR:MaxEigenvalue}.
The proof requires some results about the construction of supersequences and 
the asymptotics of sequences and their rearrangements.  
These results are of independent interest.\footnote{Hardy, Littlewood, and P\' olya include
a chapter on rearrangements of finite sequences in their book \emph{Inequalities}~\cite{hard-litt-poly88}, 
but there does not appear to have been much study of the asymptotics of rearrangements
of infinite sequences.}

\section{Asymptotics of sequence rearrangements}
Let $\N = \{1,2,3,\ldots\}$ denote the set of positive integers 
and $\N_0 = \N \cup \{ 0 \}$ the set of nonnegative integers.
Let $A = \{a_n\}_{n=1}^{\infty}$ be a sequence, and let $S(\N)$ denote the group 
of all permutations of the positive integers \N.  
For every $\sigma \in S(\N)$, the \emph{$\sigma$-rearrangement of the sequence $A$} is the sequence 
\[
A_{\sigma} = \{ a_{\sigma(n)} \}_{n=1}^{\infty}.
\]
A \emph{growth function} is a  positive, strictly increasing, continuous, 
and unbounded function $f$  defined for all real numbers $x \geq 1$.  
We write that the sequence $A$ of real numbers is \emph{asymptotic to the growth function $f$}, denoted $A \sim f,$ 
if $a_n\sim f(n)$ as $n\rightarrow \infty,$ that is, if 
$
\lim_{n\rightarrow\infty} a_n/f(n) = 1.
$

Functions $f$ and $g$, defined for all sufficiently large real numbers, are called \emph{asymptotic}, 
denoted $f \sim g,$ if $\lim_{x\rightarrow\infty} f(x)/g(x)=1.$  
A growth function $f$ is called \emph{asymptotically stable} if  $f(x+\Delta) \sim f(x)$ 
for every positive real number $\Delta.$  Polynomials are asymptotically stable.  
If $c>0,$ then the function $e^{c\sqrt{x}}$ is asymptotically stable but the exponential function $e^{cx}$ is not.

\begin{lemma}
An increasing  function $f$ is asymptotically stable if and only if  $f(x+\delta) \sim f(x)$ for some $\delta > 0.$  
\end{lemma}

\begin{proof}
Suppose that $f(x+\delta) \sim f(x)$ for some $\delta > 0.$  Let $\Delta > 0.$  If $0 < \Delta \leq \delta,$ then 
$f(x) \leq f(x+\Delta) \leq f(x+\delta)$ since $f$ is increasing, and the inequality
\[
1 \leq \frac{f(x+\Delta)}{f(x)} \leq \frac{f(x+\delta)}{f(x)}
\]
implies that $f(x+\Delta) \sim f(x).$  If $\Delta >  \delta,$ then there is a positive integer $r$ such that 
\[
r\delta \leq \Delta < (r+1)\delta.
\]
Then 
\[
f(x) \leq f(x+\Delta) \leq f(x+(r+1)\delta)
\]
and
\[
1 \leq \frac{f(x+\Delta)}{f(x)} \leq \frac{f(x+(r+1)\delta)}{f(x)} = \prod_{i=0}^r \frac{f(x+(i+1)\delta)}{f(x + i\delta)}.
\]
It follows that 
\[
1 \leq \lim_{x\rightarrow\infty} \frac{f(x+\Delta)}{f(x)} 
\leq \lim_{x\rightarrow\infty} \prod_{i=0}^r \frac{f( (x+i\delta) + \delta)}{f(x + i\delta)} = 1.
\]
This completes the proof.
\end{proof}

Note that the Lemma applies only to increasing functions.  
For example, if $c>0,$ then the positive function $f(x) = c\sin^2(\pi x)+1$ satisfies $f(x+1)\sim f(x)$ 
but $\limsup_{x\rightarrow\infty} f(x+1/2)/f(x) = c+1$ and $\liminf_{x\rightarrow\infty} f(x+1/2)/f(x) = 1/(c+1).$

Even if a sequence $A$ is asymptotic to a growth function, there can be permutations $\sigma \in S(\N)$ 
for which the rearrangement $A_{\sigma}$ is not asymptotic to a growth function.  
Here is an example.  Let $a_n=n$ for all $n\in \N$ and $f(x) = x$ for all $x \geq 1.$  
The sequence $A = \{a_n\}_{n=1}^{\infty}$ is asymptotic to the growth function $f$.  
We define the permutation $\sigma \in S(\N)$ as follows:
\[
\sigma(n) = 
\begin{cases}
n & \text{if $n \neq 2^k$ for all $k\in \N$}\\
2^{2k} & \text{if $n = 2^{2k-1}$ for some $k\in \N$}\\
2^{2k-1} & \text{if $n = 2^{2k}$ for some $k\in \N$.}
\end{cases}
\]
Since $a_{\sigma(n)} = \sigma(n)$ for all $n\in \N$, we have 
\[
\lim_{\substack{n\rightarrow\infty\\n\neq 2^k}} \frac{a_{\sigma(n)}}{n} = 1
\]
\[
\lim_{k\rightarrow\infty} \frac{a_{\sigma(2^{2k-1})}}{2^{2k-1}} = 2
\]
\[
\lim_{k\rightarrow\infty} \frac{a_{\sigma(2^{2k})}}{2^{2k}} = \frac{1}{2}.
\]
Thus, the $\sigma$-rearrangement $A_{\sigma}$ is not asymptotic to any growth function.  

Let $f$ and $g$ be asymptotically stable growth functions.  We shall prove that, for all permutations $\sigma \in S(\N)$, 
if $A \sim f$ and $A_{\sigma} \sim g,$ then $f\sim g.$

\bt
Let $A = \{a_n\}_{n=1}^{\infty}$ be a sequence of positive integers, 
and let $f$ and $g$ be asymptotically stable growth functions.  
Let $\sigma \in S(\N).$  If $A \sim f$ and $A_{\sigma} \sim g,$ 
then $f(n) \sim g(n)$ as $n\rightarrow \infty.$  
Equivalently, $a_n \sim a_{\sigma(n)}$ as $n\rightarrow \infty.$
\et

\begin{proof}
Since $A \sim f$ and $A_{\sigma} \sim g,$ it follows that for every $\varepsilon$ 
with $0 < \varepsilon < 1$  
there is a positive integer $N_0(\varepsilon)$ such that
\[
\left( \frac{1-\varepsilon}{1+\varepsilon}\right) g(N_0(\varepsilon)) \geq f(1)
\]
and
\[
(1-\varepsilon) f(n) < a_n < (1+\varepsilon) f(n)
\]
\[
(1-\varepsilon) g(n) < a_{\sigma(n)} < (1+\varepsilon) g(n)
\]
for all $n \geq N_0(\varepsilon)$.  
Choose an integer $N \geq N_0(\varepsilon).$  If $n \geq N,$ then
\beq  \label{SeqR:ineqG}
a_{\sigma(n)} > (1-\varepsilon) g(n) \geq (1-\varepsilon) g(N)
\eeq
since the growth function $g$ is increasing.  
Also, $\sigma$ is a permutation, hence, for every positive integer $j,$ 
there is a unique positive integer $i$ such that $\sigma(i)=j.$  
In particular, if $j$ is an integer such that $a_j \leq (1-\varepsilon) g(N)$ 
and if the integer $i$ satisfies $\sigma(i)=j$,  
then inequality~\eqref{SeqR:ineqG} implies that $i \leq N-1.$

If $j \geq N_0(\varepsilon)$ and $(1+\varepsilon) f(j) \leq (1-\varepsilon) g(N),$ 
then $a_j < (1-\varepsilon) g(N).$  Equivalently, if 
\beq  \label{SeqR:ineq-j}
N_0(\varepsilon) \leq j \leq f^{-1}\left( \left(\frac{1-\varepsilon}{1+\varepsilon} \right) g(N) \right)  
\eeq
then $a_j <  (1-\varepsilon) g(N).$  The number of integers $j$ that satisfy inequality~\eqref{SeqR:ineq-j} is 
\[
\left[  f^{-1}\left( \left(\frac{1-\varepsilon}{1+\varepsilon} \right) g(N) \right)  \right]-N(\varepsilon)+1
\]
and each of these $j$ is of the form $\sigma(i)$ for some positive integer $i \leq N-1.$  Therefore, 
\[
f^{-1}\left( \left(\frac{1-\varepsilon}{1+\varepsilon} \right) g(N) \right)  - N(\varepsilon) <
\left[  f^{-1}\left( \left(\frac{1-\varepsilon}{1+\varepsilon} \right) g(N) \right)  \right]-N(\varepsilon)+1  \leq N-1
\]
and
\[
\frac{1-\varepsilon}{1+\varepsilon}  <  \frac{f\left(N + N(\varepsilon) -1 \right)}{g(N)}.
\]
Since $f$ is an asymptotically stable growth function, it follows that
\[
\frac{1-\varepsilon}{1+\varepsilon}  \leq \liminf_{N\rightarrow\infty}  \frac{f\left(N + N(\varepsilon) -1 \right)}{g(N)} =  \liminf_{N\rightarrow\infty}  \frac{f\left(N\right)}{g(N)}.
\]
This inequality holds for all $\varepsilon >0$, and so 
\[
\liminf_{N\rightarrow\infty}  \frac{f\left(N\right)}{g(N)} \geq 1.
\]

Applying the same argument to the sequences $B = A_{\sigma}$ and $B_{\sigma^{-1}} = A$, where $B \sim g$ and $B_{\sigma^{-1}} \sim f,$   we obtain 
\[
\liminf_{N\rightarrow\infty}  \frac{g(N)}{f(N)} \geq 1
\]
or, equivalently,
\[
\limsup_{N\rightarrow\infty}  \frac{f(N)}{g(N)} \leq 1.
\]
It follows that
\[
\liminf_{N\rightarrow\infty}  \frac{f(N)}{g(N)} = 
\limsup_{N\rightarrow\infty}  \frac{f(N)}{g(N)} =1
\]
and so $f(n)\sim g(n)$ as $n \rightarrow \infty.$
This completes the proof.
\end{proof}

\section{A tauberian theorem for sequence rearrangements}

The sequence $A  = \{a_n\}_{n=1}^{\infty}$ is called  \emph{increasing} 
if $a_n \leq a_{n+1}$ for all $n \geq 1.$  
If $A$ is a sequence of positive integers, then there is a permutation $\sigma \in S(\N)$ 
such that the sequence $A_{\sigma}$ is increasing.  
This ``order-inducing'' permutation is unique if and only if the elements of $A$ are pairwise distinct. 
We shall prove that if $f$ is a growth function such that $A \sim f$, 
then also $A_{\sigma} \sim f, $ or, equivalently, $a_n \sim a_{\sigma(n)}$ as $n\rightarrow \infty.$  
This  ``tauberian'' result is useful in additive number theory.

\bt  \label{SeqR:theorem:RearrangeOrder}
Let $A = \{a_n\}_{n=1}^{\infty}$ be a sequence of real numbers, 
and let $f$ be a growth function such that $A \sim f.$  
If $\sigma$ is a permutation such that $a_{\sigma(n)} \leq a_{\sigma(n+1)}$ 
for all positive integers $n$, then the $\sigma$-rearrangement $A_{\sigma}$
satisfies the asymptotic relation $A_{\sigma} \sim f.$
\et

\begin{proof}
Let $0 < \varepsilon < 1.$  Since $a_n\sim f(n),$ there is a number $N_0(\varepsilon)$ such that 
\beq   \label{SeqR:BasicIneq}
(1-\varepsilon) f(n) < a_n < (1+\varepsilon) f(n)
\eeq
for all integers $n \geq N_0(\varepsilon).$  Let
\[
a^{\ast} = \max\left\{ a_k : 1 \leq k < N_0(\varepsilon) \right\}.
\]
Since the growth function $f$ increases monotonically to infinity, there is a number $N_1(\varepsilon) \geq N_0(\varepsilon)$ such that 
\[
f\left( N_1(\varepsilon)\right) > a^{\ast}.
\]
Consider an integer $n \geq N_1(\varepsilon).$ For $1 \leq k < N_0(\varepsilon)$ we have
\[
a_k \leq a^{\ast} < (1+\varepsilon) f\left( N_1(\varepsilon) \right) \leq (1+\varepsilon) f\left( n\right).
\]
For $N_0(\varepsilon) \leq k \leq n$ we have
\[
a_k < (1+\varepsilon) f\left( k\right) \leq (1+\varepsilon) f\left( n\right).
\]
We see that there are at least $n$ terms of the sequence $A$ that are strictly less than $(1+\varepsilon) f\left( n\right).$  Since the rearranged sequence $A_{\sigma} = \{a_{\sigma(n)}\}_{n=1}^{\infty}$ is monotonically increasing, it follows that
\[
a_{\sigma(n)} < (1+\varepsilon) f\left( n\right)
\]
for all $n \geq  N_1(\varepsilon).$

Similarly, if $k \geq n \geq N_1(\varepsilon),$ then
\[
a_k > (1 - \varepsilon)f(k) \geq (1 - \varepsilon)f(n)
\]
and so there are at most $n-1$ terms of the sequence $A$ that are 
less than or equal to $(1-\varepsilon) f\left( n\right).$  
Since the rearranged sequence $A_{\sigma} = \{a_{\sigma(n)}\}_{n=1}^{\infty}$ 
is monotonically increasing, it follows that
\[
a_{\sigma(n)} > (1-\varepsilon) f\left( n\right)
\]
for all $n \geq  N_1(\varepsilon).$
This proves that $a_{\sigma(n)} \sim f(n),$ or, equivalently, $A_{\sigma} \sim f.$
\end{proof}

\section{A supersequence theorem}

The \emph{lower asymptotic density} of a set $ {\mathcal{C}} $ of nonnegative integers is
\[
d_L( \mathcal{C} ) = \liminf_{x\rightarrow\infty} \frac{{ \mathcal{C}}(0,x)}{x}.
\]
The \emph{upper asymptotic density} of ${\mathcal{C}}$ is
\[
d_U( \mathcal{C} ) = \limsup_{x\rightarrow\infty} \frac{{ \mathcal{C}}(0,x)}{x}.
\]
If $d_L( \mathcal{C} ) = d_U( \mathcal{C} ),$ then the set $ \mathcal{C} $ has \emph{asymptotic density} 
\[
d( \mathcal{C} ) = d_L( \mathcal{C} ) =  \lim_{x\rightarrow\infty} \frac{{ \mathcal{C} }(0,x)}{x}.
\]

The sequence $A = \{a_k\}_{k=1}^{\infty}$ is called a \emph{subsequence} of $B = \{b_n\}_{n=1}^{\infty}$ 
if there is a strictly increasing sequence of positive integers  $\{n_k\}_{k=1}^{\infty}$ such that 
\[
a_k = b_{n_k} \qquad\text{for all $k \geq 1.$}
\]
If  $A$ is a subsequence of $B$, then $B$ is also called a \emph{supersequence} of $A$.

\begin{lemma}   \label{SeqR:lemma:smallA}
Let $f$ be a growth function such that $\lim_{x\rightarrow\infty} f(x)/x = \infty.$  
Let $A = \{a_n\}_{n=1}^{\infty}$ be a strictly increasing sequence of integers
such that $A \sim f.$  Then the set $\mathcal{A} = \{a_n : n = 1,2,\ldots \}$ has asymptotic density 0.
\end{lemma}

\begin{proof}
There is an integer $N^*$ such that $a_n > f(n)/2$ for all $n \geq N^*.$  
For every $\varepsilon > 0$ there is an integer $N_0(\varepsilon) \geq N^*$ such that
$
f(n)/n > 2/\varepsilon
$
for all $n \geq N_0(\varepsilon)$.  Let $x \geq f(N_0(\varepsilon))/2.$  There 
is a unique integer $n \geq N_0(\varepsilon)$ such that
\[
\frac{f(n)}{2} \leq x < \frac{f(n+1)}{2} < a_{n+1}.
\]
It follows that $\mathcal{A}(0,x) \leq n$ and so
\[
\frac{\mathcal{A}(0,x)}{x} \leq \frac{n}{x} \leq \frac{2n}{f(n)} < \varepsilon.
\]
Therefore,
\[
d(\mathcal{A}) = \lim_{x\rightarrow\infty} \frac{\mathcal{A}(0,x)}{x} = 0.
\]
This completes the proof.
\end{proof}

\begin{lemma}   \label{SeqR:lemma:FindSubsequence}
Let $ \mathcal{C} $ be a set of nonnegative integers of asymptotic density 1.
Let $g$ be  a growth function such that $\lim_{x\rightarrow\infty}g(x)/x = \infty.$    
There exists a sequence of integers $\{c_n\}_{n=1}^{\infty}$ 
such that $c_n \sim g(n)$ and $c_n \in  \mathcal{C} $ for all $n \in \N.$
\end{lemma}

\begin{proof}
We begin by showing that for every positive integer $t$  there is an integer $N_t$ such that 
\beq  \label{SeqR:GoodIneq}
\mathcal{C}\left( \left(1-\frac{1}{t} \right) g(n),  \left(1+ \frac{1}{t} \right) g(n) \right)
  \geq \frac{g(n)}{t}
\qquad\text{for all $n \geq N_t$,}
\eeq
where $\mathcal{C}(y,x)$ is the counting function of the set $\mathcal{C}.$
If not, then for some $t$ there are infinitely many integers $n$ for which 
\beq  \label{SeqR:BadIneq}
\mathcal{C}\left( \left(1-\frac{1}{t} \right) g(n),  \left(1+ \frac{1}{t} \right) g(n) \right) < \frac{g(n)}{t}.
\eeq
If $n$ satisfies inequality~\eqref{SeqR:BadIneq}, then  
\[
\mathcal{C}\left( 0, \left(1+ \frac{1}{t} \right)g(n)\right) <  g(n) +1.
\]
It follows that
\[
d_L( \mathcal{C} ) = \liminf_{n\rightarrow\infty} \frac{ \mathcal{C}(0,n)}{n} \leq \frac{t}{t+1}.
\]
which contradicts the fact that the  set $ {\mathcal{C}} $ has asymptotic density 1.    
Therefore, inequality~\eqref{SeqR:GoodIneq} holds for all $t\geq 1.$  
Since $g(n)/n$ tends to infinity, we can also choose the positive integers $N_t$ so that 
\[
\frac{g(n)}{n} > t \qquad\text{for all $n \geq N_t$}
\]
and
\[
N_t < N_{t+1} \qquad\text{for all $t \geq 1$.}
\]
If $n \geq N_t,$ then the interval 
\[
\left[ \left(1-\frac{1}{t} \right) g(n),  \left(1+ \frac{1}{t} \right) g(n) \right] 
\]
contains at least $g(n)/t > n$ elements of the set ${\mathcal{C}}$.
In particular, for $t=1,$ the interval $\left[ 0, 2 g(N_1) \right]$ 
contains more than $N_1$ elements of $ \mathcal{C} $.  
We choose distinct positive integers $c_1,\ldots,c_{N_1}$ in the set $ \mathcal{C}  \cap \left[ 0, 2 g(N_1) \right]$.

Let $n' >N_1$ and suppose that we have constructed a finite sequence 
of pairwise distinct integers $\{c_n\}_{n=1}^{n'-1}$ such that 
\[
c_n \in  \mathcal{C}  \cap \left[ \left(1-\frac{1}{t} \right) g(n),  \left(1+ \frac{1}{t} \right) g(n) \right]
\]
for all integers $n$ such that  $N_1 \leq n < n'$ and $N_t \leq n < N_{t+1}.$  Choose the positive integer $t'$ so that $N_{t'}\leq n' < N_{t'+1}.$  
Since the interval
\[
\left[ \left(1-\frac{1}{t'} \right) g(n'),  \left(1+ \frac{1}{t'} \right) g(n') \right] 
\]
contains at least $g(n')/t' \geq n'$ elements of ${\mathcal{C}}$, 
the pigeon hole principle implies that we can choose an integer $c_{n'}$ in this interval 
such that $c_n \neq c_{n'}$ for all integers $n < n'.$
It follows by induction that there is an infinite sequence $\{c_n\}_{n=1}^{\infty}$ of pairwise distinct integers such that 
\[
c_n \in  \mathcal{C}  \cap \left[ \left(1-\frac{1}{t} \right) g(n),  \left(1+ \frac{1}{t} \right) g(n) \right]
\]
for all integers $n \geq N_t$.  Equivalently,
\[
\left(1-\frac{1}{t} \right) g(n) \leq c_n \leq  \left(1+ \frac{1}{t} \right) g(n)
\]
for all $n \geq N_t,$  and so the sequence $\{c_n\}_{n=1}^{\infty}$ satisfies the asymptotic relation $c_n \sim g(n)$.  
This completes the proof.
\end{proof}

The following result about supersequences will be used to describe the additive spectrum.

\begin{theorem}   \label{SeqR:theorem:FundamentalSuper}
Let $f$ be a growth function, and let $A = \{a_k\}_{k=1}^{\infty}$ 
be a strictly increasing sequence of integers such that $A \sim f$. 
Let $g$ be an asymptotically stable growth function such that 
$\lim_{x\rightarrow\infty} g(x)/x = \infty.$  
If $g(x) \leq f(x)$ and $g^{-1}f(x+1) - g^{-1}f(x) \geq 1$ for all $x\geq 1,$ 
then there exists a strictly increasing sequence of positive integers  $B = \{b_n\}_{n=1}^{\infty}$ such that 
$B \sim g$ and $B$ is a supersequence of $A$.
\end{theorem}

\begin{proof}
We begin by proving that there is a  strictly increasing sequence  
of positive integers 
$\{n_k\}_{k=1}^{\infty}$ such that $g(n_k) \sim f(k).$  
Since $g(x) \leq f(x)$ for $x \geq 1$ and $g$ is continuous, it follows that 
\[
[g(1),\infty) \supseteq [f(1),\infty)
\]
and so the range of $g$ contains the range of $f$.  
For every positive integer $k$ we define the real number $t_k = g^{-1}f(k)$ 
and the positive integer  
$
n_k = [t_k] = t_k.
$
The inequality $g(x) \leq f(x)$ implies that $x \leq g^{-1}f(x),$ 
and so $t_k \geq n_k \geq k.$ 
Since
\[
t_{k+1} - t_k =  g^{-1}f(k+1) -  g^{-1}f(k) \geq 1
\]
it follows that
\[
n_{k+1} > t_{k+1} - 1 \geq t_k \geq n_k
\]
and so $\{n_k\}_{k=1}^{\infty}$ is a strictly increasing sequence of positive integers.    
Since $g(x-1) \sim g(x)$ and $g(t_k -1) < g(n_k) \leq g(t_k),$ 
it follows that $g(n_k) \sim g(t_k) = f(k).$ 

Let $ {\mathcal{A}}  = \{a_k : k = 1,2,\ldots \}.$
Since $g(x) \leq f(x)$ and $\lim_{x\rightarrow\infty} g(x)/x = \infty,$
it follows that $\lim_{x\rightarrow\infty} f(x)/x = \infty.$
By Lemma~\ref{SeqR:lemma:smallA}, the set $\mathcal{A}$ has asymptotic density 0, 
and so the set $\mathcal{C} = \N_0 \setminus \mathcal{A}$ has asymptotic density 1.  
By Lemma~\ref{SeqR:lemma:FindSubsequence}, the set ${\mathcal{C}}$ contains a subsequence $\{c_n\}_{n=1}^{\infty}$ 
such that $c_n \sim g(n).$  Define the sequence $B' = \{b'_n\}_{n=1}^{\infty}$ by
\[
b'_n = \begin{cases}
c_n & \text{if $n \neq n_k$ for all $k\in \N$} \\
a_k & \text{if $n = n_k.$}
\end{cases}
\]
The elements of the sequence $B'$ are pairwise distinct because the sets 
${\mathcal{A}}$ and ${\mathcal{C}}$ are disjoint.  Moreover,
\[
\lim_{\substack{n\rightarrow\infty \\ n \neq n_k}} \frac{b'_n}{g(n)}
= \lim_{\substack{n\rightarrow\infty \\ n \neq n_k}} \frac{c_n}{g(n)} = 1
\]
and
\[
\lim_{k\rightarrow\infty}  \frac{b'_{n_k}}{g(n_k)} 
= \lim_{k\rightarrow\infty} \frac{a_k}{f(k)}\frac{f(k)}{g(n_k)} = 1.
\]
Thus, $B' \sim g.$  
However, the terms of the sequence $B'$ are not necessarily strictly increasing.  
Choose a permutation $\sigma \in S(\N)$ such that the rearranged sequence
$B = B'_{\sigma}$ is strictly increasing.  By Theorem~\ref{SeqR:theorem:RearrangeOrder}, 
we have $B \sim g.$  This completes the proof.
\end{proof}

\section{Approximating powers of 3 by powers of 2, and other asymptotic impossibilities}

For every real number $x$, let $\langle x \rangle = x - [x]$ 
denote the \emph{fractional part} of $x$.

In this section we show that the supersequence theorem 
(Theorem~\ref{SeqR:theorem:FundamentalSuper}) is false 
if we omit the condition that the growth function $g$ is asymptotically stable.
We begin with arithmetically interesting special case of $f(x) = 3^x$ and $g(x) = 2^x.$
The following result is equivalent to the statement  
that it is impossible to approximate powers of 3 by powers of 2.  
The proof uses the fact that if $\vartheta$ is an irrational number, 
then the sequence of fractional parts $\{ \langle k\vartheta \rangle \}_{k=1}^{\infty}$  
is dense in the interval $(0,1).$

\begin{theorem}  \label{SeqR:theorem:PowersOf3and2}
Let $A = \{a_k\}_{k=1}^{\infty}$ be a strictly increasing 
sequence of integers such that $A \sim 3^x.$ 
There does not exist a strictly increasing sequence of integers 
$B = \{b_n\}_{n=1}^{\infty}$ such that $B \sim 2^x$ 
and $B$ is a supersequence of $A$.
\end{theorem}

\begin{proof}
Suppose that $B = \{b_n\}_{n=1}^{\infty}$ is a supersequence of $A$ such that $B \sim 2^x.$  
Then there is a strictly increasing sequence of integers $\{n_k\}_{k=1}^{\infty}$ such that
\[
b_{n_k} = a_k
\]
for all $k \geq 1$.  It  follows that
\[
\lim_{k\rightarrow\infty} \frac{2^{n_k}}{3^k} = 
 \lim_{k\rightarrow\infty} \frac{2^{n_k}}{b_{n_k}}\frac{a_k}{3^k}=1
\]
and so, for every $\varepsilon$ with $0 < \varepsilon < 1/2$,   
there exists an integer $K(\varepsilon)$ such that 
\[
(1-\varepsilon)3^k < 2^{n_k} < (1+\varepsilon)3^k
\]
for all $k \geq K(\varepsilon).$  Taking logarithms, we obtain
\[
-1 < -\frac{\log(1+\varepsilon)}{\log 2} < k\left(\frac{\log 3}{\log 2}\right) - n_k 
< \frac{\log(1-\varepsilon)^{-1}}{\log 2} < 1
\]
and so the fractional part of $ k\log 3/\log 2$ satisfies 
\[
\left\langle  k\left(\frac{\log 3}{\log 2}\right) \right\rangle \in \left( 0, \frac{\log(1-\varepsilon)^{-1}}{\log 2}\right) 
\cup \left( 1 - \frac{\log(1+\varepsilon)}{\log 2}, 1 \right).  
\]
If we choose 
\[
0 < \varepsilon < 1 - 2^{-1/4}
\]
then the fractional part of $k\log 3 / \log 2$ satisfies 
\[
\left\langle  k\left(\frac{\log 3}{\log 2}\right) \right\rangle \in \left( 0, \frac{1}{4}\right) 
\cup \left(\frac{3}{4},1 \right)
\]
for all integers $k \geq K(\varepsilon).$  This is impossible, since $\log 3/\log 2$ is irrational and the sequence
$\{ \left\langle k\log 3 / \log 2 \right\rangle  \}_{k=1}^{\infty}$ is dense in $(0,1).$  This completes the proof.
\end{proof}

\begin{theorem}               
Let $u$ and $v$ be positive integers with $u > v.$  
Let $A = \{a_k\}_{k=1}^{\infty}$ be a strictly increasing sequence 
of positive integers such that $A \sim u^x.$  There exists a strictly increasing 
sequence of positive integers $B = \{b_n\}_{n=1}^{\infty}$ 
such that (i) $B \sim v^x$ and (ii) $B$ is a supersequence of $A$ 
if and only if $u = v^r$ for some integer $r \geq 2.$
\end{theorem}

\begin{proof}
If $u = v^r,$ then we simply let $b_n = v^n$ and $n_k = rk$ for all $k \geq 1.$

If there exists a sequence $B = \{b_n\}_{n=1}^{\infty}$ such that $B \sim v^x$ 
and $B$ is a supersequence of $A$, then there is a strictly increasing sequence of integers
$\{n_k\}_{k=1}^{\infty}$ such that $b_{n_k} = a_k$ for all $k \in \N.$  
It follows that 
\[
\lim_{k\rightarrow\infty} \frac{v^{n_k}}{u^k} 
= \lim_{k\rightarrow\infty} \frac{v^{n_k}}{b_{n_k}}\frac{a_k}{u^k} = 1.
\] 
Suppose that $u \neq v^r$ for all integers $r \geq 2.$ 
There are two cases.  In the first case, $\log u/\log v$ is rational.  
Then there exist relatively prime positive integers $r$ and $s$ such that $1 < s < r$ and
$\log u/\log v = r/s,$ or, equivalently, $u^s = v^r.$   
Choose an integer $\ell$ such that $1 \leq \ell \leq s-1$ and $\ell r \equiv 1 \pmod{s}.$
Let $m_{j} = n_{j s + \ell}$ for all positive integers $j.$
Then
\[
\lim_{j\rightarrow\infty} \frac{v^{m_{j}}}{u^{j s + \ell}} = 1.
\]
It follows that for every $\varepsilon > 0$ there is an integer $J(\varepsilon)$ such that 
\[
(1 - \varepsilon) u^{j s + \ell} < v^{m_{j}} < (1 + \varepsilon) u^{j s + \ell}
\]
for all $j \geq J(\varepsilon).$  Taking logarithms and rewriting the inequalities, we obtain
\[
-\frac{\log (1 + \varepsilon)}{\log v} 
< \frac{( j s + \ell ) \log u}{\log v} - m_{j} 
< \frac{\log (1 - \varepsilon)^{-1}}{\log v}.
\]
Choosing $0 < \varepsilon < 1 - v^{-1/s},$ we obtain  
\[
0 < \frac{\log(1+\varepsilon)}{\log v} < \frac{\log(1-\varepsilon)^{-1}}{\log v} < \frac{1}{s}
\]
Since $\log u/\log v = r/s$ and $\ell r \equiv 1 \pmod{s}$, 
it follows that for every integer $j \geq J(\varepsilon)$ there is an integer $w_{j}$ such that 
\[
-\frac{1}{s} < -\frac{\log (1 + \varepsilon)}{\log v} 
< \frac{1}{s} + w_{j} 
< \frac{\log (1 - \varepsilon)^{-1}}{\log v} < \frac{1}{s}.
\]
This is impossible for $s > 1.$

In the second case, $\log u/\log v$ is irrational, the sequence of fractional parts
$\{ \langle k\log u/\log v\rangle\}_{k=1}^{\infty}$ is dense in the interval $(0,1),$ 
and the argument proceeds as in Theorem~\ref{SeqR:theorem:PowersOf3and2}.
This completes the proof.  
\end{proof}

The function $g$ has \emph{exponential growth} if $\liminf_{x\rightarrow\infty} g(x+\delta)/g(x) > 1$ 
for some $\delta > 0.$  For $c>0,$ the exponential function $g(x) = e^{cx}$ 
satisfies $g(x+\delta)/g(x) = e^{c\delta} > 1$ for all $\delta > 0.$  

We need the following simple interpolation result.

\begin{lemma}  \label{SeqR:lemma:SuperFunction}
Let $g$ be a growth function and let $\{ \lambda_k \}_{k=1}^{\infty}$ be a strictly increasing sequence 
of real numbers such that $\lambda_k > g(k)$ for all $k \geq 1.$  There exists a growth function
$f$ such that $f(x) > g(x)$ for all $x \geq 1$ and $f(k) = \lambda_k$ for all $k \geq 1.$
\end{lemma}

\begin{proof}
It suffices to construct $f$ on each interval $[k,k+1].$  We define the number
\[
\mu = \min\left( \lambda_k - g(k), \lambda_{k+1} - g(k+1) \right) > 0
\]
and the strictly increasing, continuous functions 
\[
f_1(x) = g(x)+ \mu
\]
and
\[
f_2(x) = \left( \lambda_{k+1} - \lambda_k \right) (x-k) + \lambda_k.
\]
The function
\[
f(x) = \max(f_1(x),f_2(x))
\]
satisfies the requirements of the Lemma.
\end{proof}

\begin{theorem}
Let $g$ be a growth function such that there is a strictly 
increasing sequence of integers $\{m_k\}_{k=1}^{\infty}$ such that 
\[
\liminf_{k\rightarrow\infty} \frac{g(m_k+1/2)}{g(m_k)} > 1 
\]
and
\[
\liminf_{k\rightarrow\infty} \frac{g(m_k+1)}{g(m_k + 1/2)} > 1.
\]  
There is a growth function $f$ such that $f(x) > g(x)$ for all $x\geq 1$ 
and there is a strictly increasing sequence $A$ of positive integers with $A \sim f$ 
such that there 
does not exist a strictly increasing sequence $B$ of positive integers 
with the properties that $B\sim g$ and $B$ is a supersequence of $A$.
\end{theorem}

\begin{proof}
Since $\lim_{x\rightarrow\infty} g(x) = \infty,$  
by choosing a subsequence of $\{m_k\}_{k=1}^{\infty}$, we can assume 
without loss of generality that 
\beq         \label{SeqR:gm-ineq}
g(m_{k+1}+1/2) - g(m_k + 1/2) \geq 1.
\eeq
The sequence  $\{m_k\}_{k=1}^{\infty}$ is strictly increasing, 
and so $m_k \geq k$ for all $k \geq 1.$
Define the sequence $\{ {\lambda}_k \}_{k=1}^{\infty}$ by
$\lambda_k = g(m_k + 1/2)$ for all $k \geq 1.$  Since the growth function
$g$ is strictly increasing, we have
\[
g(k) \leq g(m_k) < g(m_k + 1/2) = \lambda_k < g(m_{k+1}+1/2) = \lambda_{k+1}.
\]
The sequence $\{\lambda_k\}_{k=1}^{\infty}$  satsifies the conditions of Lemma~\ref{SeqR:lemma:SuperFunction}.
Let $f$ be a growth function such that $f(k) = \lambda_k$ for all integers $k \geq 1$ 
and $f(x) > g(x)$ for all real numbers $x \geq 1.$  

Let $A = \{a_k\}_{k=1}^{\infty}$ be the sequence of integers defined by 
\[
a_k = [f(k)] = [ g(m_k + 1/2)].
\]
Condition~\eqref{SeqR:gm-ineq} implies that $a_k < a_{k+1}$ for all $k\geq 1.$
Moreover, $A \sim f.$  Suppose that $B =  \{b_n\}_{n=1}^{\infty}$ is a strictly 
increasing sequence of positive integers
such that $B \sim g$ and $B$ is a supersequence of $A$.  
Then there exists a strictly increasing sequence $\{n_k\}_{k=1}^{\infty}$ 
of positive integers such that $b_{n_k} = a_k$ for all positive integers $k$.  
This implies that 
\[
\lim_{k\rightarrow\infty} \frac{f(k)}{g(n_k)} = \lim_{k\rightarrow\infty}  \frac{f(k)}{a_k} \frac{b_{n_k}}{g(n_k)} =1.
\]
We shall prove that this is impossible.

Let $\{ n_k \}_{k=1}^{\infty}$ be the strictly increasing sequence of integers  
such that $g(n_k) \sim f(k).$  Either $n_k \leq m_k$ for infinitely many $k$, 
or $n_k \geq m_k + 1$ for infinitely many $k.$
In the first case, 
\[
\liminf_{k\rightarrow\infty} \frac{f(k)}{g(n_k)} \geq  
 \liminf_{k\rightarrow\infty} \frac{g(m_k+1/2)}{g(m_k)} > 1 
\]
which contradicts the asymptotic relation $g(n_k) \sim f(k).$  In the second case, 
\[
\limsup_{k\rightarrow\infty} \frac{f(k)}{g(n_k)} \leq \limsup_{k\rightarrow\infty} \frac{g(m_{k}+1/2)}{g(m_k +1)} < 1
\]
which also contradicts  $g(n_k) \sim f(k).$
This completes the proof.
\end{proof}

\section{The spectrum of bases in additive number theory}
\begin{theorem}    \label{SeqR:theorem:eigenvalue}
Let $h \geq 2.$  
If $\alpha$ is an additive eigenvalue of order $h$ and $0 < \beta < \alpha,$ 
then $\beta$ is also an additive eigenvalue of order $h$.
\end{theorem}

\begin{proof}
Let $f(x) = \alpha x^h$ and $g(x) = \beta x^h.$  Then $f$ and $g$ are 
asymptotically stable growth functions such that $g(x) < f(x)$ for all $x \geq 1,$  
and 
\[
\lim_{x\rightarrow\infty} \frac{g(x)}{x} = \lim_{x\rightarrow\infty} \beta x^{h-1} = \infty.
\]
Moreover, the function 
\[
g^{-1}f(x) = \left( \frac{\alpha}{\beta} \right)^{1/h} x
\]
satisfies the identity
\[
g^{-1}f(x+1) - g^{-1}f(x) = \left( \frac{\alpha}{\beta} \right)^{1/h} > 1
\]
for all $x \geq 1.$

If $\alpha$ is an additive eigenvalue of order $h$, then there is  an asymptotic basis 
 $A = \{a_k\}_{k=1}^{\infty}$  of order $h$ such that $A \sim f,$ that is, $a_k \sim \alpha k^h.$
Applying Theorem~\ref{SeqR:theorem:FundamentalSuper} to the sequence $A$ and the growth functions $f$ and $g$, 
we obtain a supersequence $B = \{b_n\}_{n=1}^{\infty}$ of $A$ such that $B \sim g,$ that is, $b_n \sim \beta n^h.$  
Since $B$ contains $A$, it follows that $B$ is also an asymptotic basis of order $h$, and so $\beta$ is an additive eigenvalue of order $h$.  This completes the proof.
\end{proof}

Theorem~\ref{SeqR:theorem:eigenvalue} immediately implies the following result.

\begin{theorem}
For every $h \geq 2,$ the additive spectrum  $\mathcal{N}(h)$ is an interval 
of the form $(0, \eta_h)$ or $(0, \eta_h]$ with $\eta_h \leq 1/h!.$  
\end{theorem}

It is an open problem to compute the number $\eta_h$ and to determine if $\eta_h$ is an additive eigenvalue of order $h$.

\providecommand{\bysame}{\leavevmode\hbox to3em{\hrulefill}\thinspace}
\providecommand{\MR}{\relax\ifhmode\unskip\space\fi MR }
\providecommand{\MRhref}[2]{%
  \href{http://www.ams.org/mathscinet-getitem?mr=#1}{#2}
}
\providecommand{\href}[2]{#2}

\end{document}